\documentclass[a4paper]{amsart}
\usepackage{graphicx,multirow,array,amsmath,amssymb}

\theoremstyle{plain}
\newtheorem{thm}{Theorem}[section]
\newtheorem{lem}[thm]{Lemma}
\newtheorem{prop}[thm]{Proposition}

\theoremstyle{definition}
\newtheorem{defn}[thm]{Definition}

\begin{document}

\title{Petal number of torus knots of type \boldmath$(r,r+2)$}

\author[H. J. Lee]{Hwa Jeong Lee}
\address{Department of Mathematics Education, Dongguk University--Gyeongju, 123 Dongdae-ro, Gyeongju, Gyeongbuk, South Korea}
\email{hjwith@dongguk.ac.kr}
\author[G. T. Jin]{Gyo Taek Jin}
\address{Department of Mathematical Sciences, Korea Advanced Institute of Science and Technology, Daejeon 34141, Korea}
\email{trefoil@kaist.ac.kr}

\thanks{keywords: torus knot;  petal projection;  petal number;  grid diagram.}


\dedicatory{This work is dedicated to the memory of Vaughan Jones.}

\begin{abstract}
Let $r$ be an odd integer, $r\ge3$. Then the petal number of the torus knot of type $(r,r+2)$ is equal to $2r+3$.
\end{abstract}

\maketitle

\section{introduction}

A knot diagram is a regular projection of a knot that has relative height information at each of the double points. In~\cite{Adams2013, Adams2014}, Adams introduced an $n$-crossing in a knot or link projection (also known as a multi-crossing when $n$ is not specified), which is a point such that $n$ strands pass straight through it and the relative heights of the strands are specified.
For each $n \geq 2$, every knot and link has a projection such that all crossings are $n$-crossings~\cite{Adams2013}. In~\cite{Adams2015_2}, it is proved that every knot has a projection with a single multi-crossing. 

\begin{defn} 
A {\em petal projection} of a knot $K$ is a projection of $K$ with a single multi-crossing such that there are no nesting loops. The {\em petal number}, $p(K)$, is the minimum number of loops among all petal projections of $K$, or equivalently, the minimum number of strands passing through the single multi-crossing. 
Suppose we have a petal projection with $n$ loops. 
We label the strands passing through the $n$-crossing with $1,2,\ldots, n$ from top to bottom.
From one end of the top strand we read the labels clockwise half way around the crossing. 
This sequence of labels is called the {\em petal permutation} of the petal projection.  
\end{defn}

\begin{figure}[h]
\begin{center}
\includegraphics[scale=0.52]{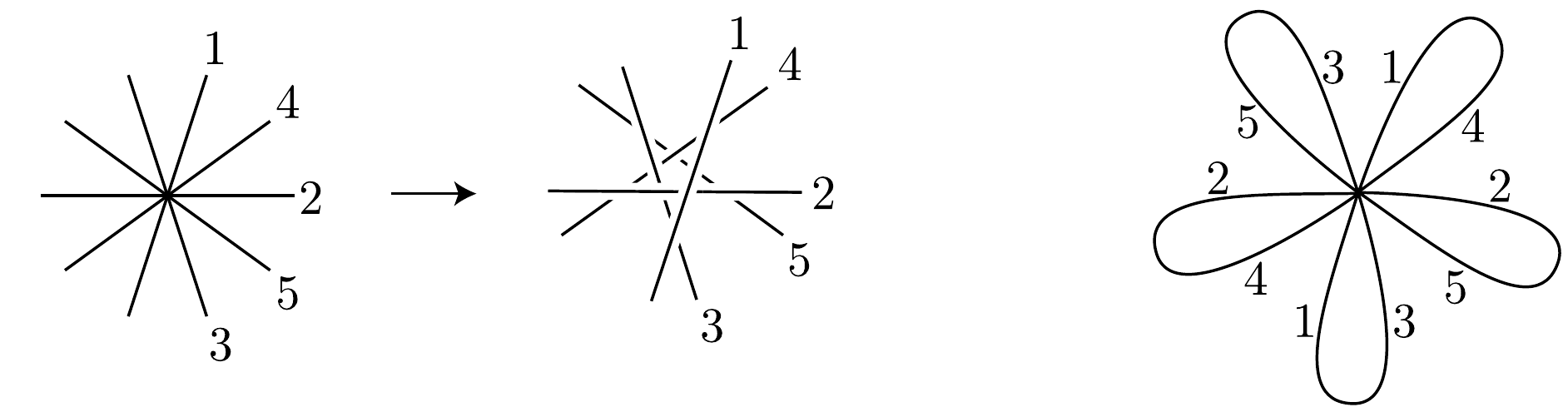}
\caption{A $5$-crossing and a petal projection of a trefoil knot.}
\label{fig:5-crossing}
\end{center}
\end{figure}

The left side of Figure~\ref{fig:5-crossing} shows an example of a $5$-crossing.
The right one shows a petal projection of the left-handed trefoil with the petal permutation, $(1,4,2,5,3)$.

For a pair of relatively prime  integers  $r,s$ with $2\leq r<s$, let $T_{r,s}$ denote the torus knot of type $(r,s)$. 
In~\cite{Adams2015_2}, Adams et al. ~showed that $p(T_{r,r+1})=2r+1$. 
In this paper, we consider torus knots of type $(r,r+2)$ and prove the following theorem.

\begin{thm}\label{thm:main} 
Let $r$ be an odd integer and $r\ge3$. Then
$p\left(T_{r,r+2}\right)=2r+3$.
\end{thm}

\section{Petal number and arc index}\label{sec:arc_index}
Given a petal projection of a knot $K$ with $n$ loops,  we can embed the loops on $n$ half planes with polar angles $0<\theta_1<\theta_2<\cdots<\theta_n<2\pi$ one by one.
Such an embedding of one arc on one half plane is called an arc presentation.  The minimal number of half planes needed for an arc presentation is called the arc index, denoted by $\alpha(K)$~\cite{C}.
It is obvious that $p(K)\ge\alpha(K)$ for any knot $K$.

A grid diagram is a knot diagram which is composed of horizontal sticks and vertical sticks such that at every crossing the vertical stick crosses over the horizontal stick.
The arc index $\alpha(K)$ is the same as the minimal number of vertical sticks among all grid diagrams of $K$.

\begin{figure}[h]
\includegraphics[width=0.12\textwidth]{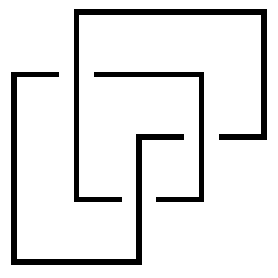}
\caption{Minimal grid diagram of a left-handed trefoil}
\end{figure}

\begin{prop}\cite{Adams2015_2}\label{prop:petal_l_bound}
Let $K$ be a nontrivial knot. Then 
$$p(K)\geq 
\begin{cases}
\alpha(K) \qquad &{\text{if \ }} \alpha(K) {\text{ is odd}},\\
\alpha(K)+1 \qquad &{\text{if \ }} \alpha(K) {\text{ is even}}.
\end{cases}$$
\end{prop}

\begin{prop}\cite{EH}\label{prop:arc_torus}
Let $r$ and $s$ be relatively prime  integers such that $2\le r<s$. Then
$$\alpha\left(T_{r,s}\right)=r+s.$$
\end{prop}

\section{Proof of the main theorem}\label{sec:proof}
By Propositions \ref{prop:petal_l_bound} and \ref{prop:arc_torus}, we know that
$p(T_{r,r+2})\ge2r+3$ if $r\ge3$ is an odd integer. 
Now we show that $p(T_{r,r+2})\le2r+3$.

Let $r=2n+1$ for a positive integer $n$ and let $\sigma_i, i=1,2,\ldots,2n$, denote the standard generators of  the braid group $B_r$ of $r$ strings. 

\centerline{\includegraphics[width=0.45\textwidth]{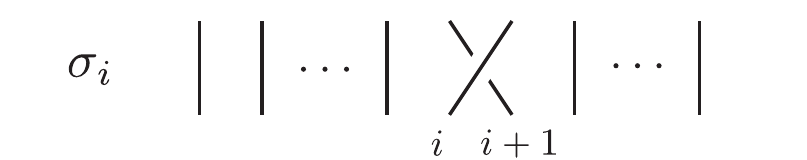}}

\noindent
Let $\Delta=\Delta_r$ denote the positive half-twist of $r$ strings and let 
$$\tau=\tau_r=\prod_{i=1}^{2n}\sigma_i=\sigma_1\sigma_2\cdots\sigma_{2n}.$$

We introduce some braid relations needed to prove Lemma~\ref{lem:conjugate}.
\begin{equation}
\begin{aligned}
\sigma_i\sigma_j\sigma_i^\varepsilon&=\sigma_j^\varepsilon\sigma_i\sigma_j 
&&|i-j|=1,\ \varepsilon=\pm1,\ 1\le i,j\le 2n\\
\sigma_i\sigma_j^\varepsilon&=\sigma_j^\varepsilon\sigma_i  &&|i-j|>1,\ \varepsilon=\pm1,\ 1\le i,j\le 2n
\end{aligned}
\end{equation}
The above  are extended from the standard defining relations of $B_r$.
\begin{equation}
\sigma_i^\varepsilon\Delta^2=\Delta^2\sigma_{i}^\varepsilon \qquad
\varepsilon=\pm1,\ 1\le i \le 2n
\end{equation}
In fact, the positive  full-twist $\Delta^2$ commutes with any braid.
\begin{equation}
\sigma_i^\varepsilon\tau=\tau\sigma_{i-1}^\varepsilon\qquad
 \varepsilon=\pm1,\ 2\le i\le 2n
 \end{equation}
An example of the above is shown below.  Notice that we read braid words from the bottom.

\centerline{\includegraphics[width=0.35\textwidth]{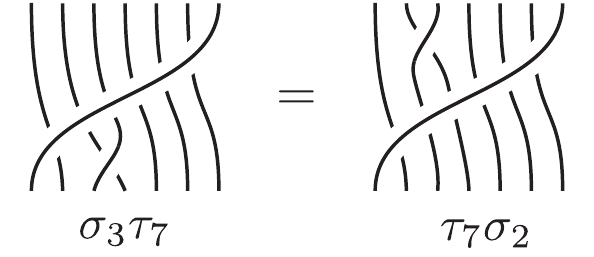}}

\begin{lem}\label{lem:conjugate}
In $B_r$, the  braid
$$\Delta^2\tau
(\sigma_{n+1}\sigma_{n+2}\cdots\sigma_{2n})(\sigma_{2n}\sigma_{2n-1}\cdots\sigma_{n+1})$$
 is a conjugate of $\Delta^2\tau^2$.
\end{lem}

\begin{proof}
In the cases $n=1$ and $n=2$, the lemma is verified by the following computation.
$$\begin{aligned}
(\sigma_2^{-1})\Delta^2\tau^2(\sigma_2)
&=\Delta^2\tau\sigma_1^{-1}\tau(\sigma_2)\\
&=\Delta^2\tau(\sigma_2)(\sigma_2)\\
(\sigma_3^{-1}\sigma_4^{-1}\sigma_2^{-1})\Delta^2\tau^2(\sigma_2\sigma_4\sigma_3)
&=\Delta^2\tau(\sigma_2^{-1}\sigma_3^{-1}\sigma_1^{-1})\tau(\sigma_2\sigma_4\sigma_3)\\
&=\Delta^2\tau(\sigma_2^{-1}\sigma_1^{-1})\tau\sigma_2^{-1}(\sigma_2\sigma_4\sigma_3)\\
&=\Delta^2\tau(\sigma_3\sigma_4)(\sigma_4\sigma_3)
\end{aligned}$$
Let $n\ge3$ and let 
$$\begin{aligned}
c_k&=\sigma_{k+1}\sigma_{k+3}\cdots\sigma_{2n-k+1}, \quad k=1,\ldots,n\\
\beta_0&=\Delta^2\tau^2\\
\beta_1&=c_1^{-1}\beta_0c_1\\
\beta_k&=c_k^{-1}\beta_{k-1}c_k,\quad k=2,\ldots,n.
\end{aligned}$$
We compute $\beta_1,\beta_2,\ldots,\beta_n$ inductively.
$$\begin{aligned}
\beta_1&=(\sigma_{2n}^{-1}\sigma_{2n-2}^{-1}\cdots\sigma_{2}^{-1})
\Delta^2\tau^2(\sigma_{2}\sigma_{4}\cdots\sigma_{2n})\\
&=\Delta^2\tau\sigma_{1}^{-1}\tau(\sigma_{2n-2}^{-1}\cdots\sigma_{2}^{-1})(\sigma_{2}\sigma_{4}\cdots\sigma_{2n})\\
&=\Delta^2\tau(\sigma_1^{-1})\tau(\sigma_{2n})\\
\beta_2&=(\sigma_{2n-1}^{-1}\sigma_{2n-3}^{-1}\cdots\sigma_{3}^{-1})\Delta^2\tau(\sigma_1^{-1})\tau(\sigma_{2n})(\sigma_{3}\sigma_{5}\cdots\sigma_{2n-1})\\
&=\Delta^2\tau(\sigma_{2n-2}^{-1}\sigma_{2n-4}^{-1}\cdots\sigma_{2}^{-1})
\sigma_1^{-1}\tau(\sigma_{2n})(\sigma_{3}\sigma_{5}\cdots\sigma_{2n-1})\\
&=\Delta^2\tau(\sigma_2^{-1}\sigma_1^{-1})\tau(\sigma_{2n})(\sigma_{2n-3}^{-1}\cdots\sigma_{3}^{-1})(\sigma_{3}\sigma_{5}\cdots\sigma_{2n-1})\\
&=\Delta^2\tau(\sigma_2^{-1}\sigma_1^{-1})\tau(\sigma_{2n}\sigma_{2n-1})
\end{aligned}$$
Suppose that 
$$\beta_{k}=\Delta^2\tau(\sigma_k^{-1}\cdots\sigma_2^{-1}\sigma_1^{-1})\tau(\sigma_{2n}\sigma_{2n-1}\cdots\sigma_{2n-k+1}),
$$
for $k<n$. Then
$$\begin{aligned}
\beta_{k+1}&=(\sigma_{2n-k}^{-1}\sigma_{2n-k-2}^{-1}\cdots\sigma_{k+2}^{-1})
\Delta^2\tau(\sigma_k^{-1}\cdots\sigma_2^{-1}\sigma_1^{-1})\tau\\
&\qquad\qquad\qquad\qquad (\sigma_{2n}\sigma_{2n-1}\cdots\sigma_{2n-k+1})(\sigma_{k+2}\sigma_{k+4}\cdots\sigma_{2n-k})\\
&=\Delta^2\tau(\sigma_{2n-k-1}^{-1}\sigma_{2n-k-3}^{-1}\cdots\sigma_{k+1}^{-1})(\sigma_k^{-1}\cdots\sigma_2^{-1}\sigma_1^{-1})\tau\\
&\qquad\qquad\qquad\qquad (\sigma_{2n}\sigma_{2n-1}\cdots\sigma_{2n-k+1})(\sigma_{k+2}\sigma_{k+4}\cdots\sigma_{2n-k})\\
&=\Delta^2\tau(\sigma_{k+1}^{-1}\sigma_k^{-1}\cdots\sigma_2^{-1}\sigma_1^{-1})\tau
(\sigma_{2n-k-2}^{-1}\sigma_{2n-k-4}^{-1}\cdots\sigma_{k+2}^{-1})\\
&\qquad\qquad\qquad\qquad (\sigma_{2n}\sigma_{2n-1}\cdots\sigma_{2n-k+1})(\sigma_{k+2}\sigma_{k+4}\cdots\sigma_{2n-k})\\&=\Delta^2\tau(\sigma_{k+1}^{-1}\sigma_k^{-1}\cdots\sigma_2^{-1}\sigma_1^{-1})\tau
(\sigma_{2n}\sigma_{2n-1}\cdots\sigma_{2n-k+1}\sigma_{2n-k})\\
\end{aligned}$$
Inductively, we obtain
$$\begin{aligned}\beta_n&=\Delta^2\tau(\sigma_{n}^{-1}\sigma_{n-1}^{-1}\cdots\sigma_2^{-1}\sigma_1^{-1})\tau(\sigma_{2n}\sigma_{2n-1}\cdots\sigma_{n+1})\\
&=\Delta^2\tau(\sigma_{n+1}\sigma_{n+2}\cdots\sigma_{2n})(\sigma_{2n}\sigma_{2n-1}\cdots\sigma_{n+1})
\end{aligned}$$
Since $\beta_n$ is a conjugate of $\beta_0$, we are done.\end{proof}

Notice that $\tau^r$ is also a positive full-twist of $r$ strings.  Therefore
$\Delta^2\tau^2=\tau^{r+2}$,  showing that the closure of $\Delta^2\tau^2$ is the torus knot $T_{r,r+2}$.

\begin{figure}[h]
\includegraphics[width=0.7\textwidth]{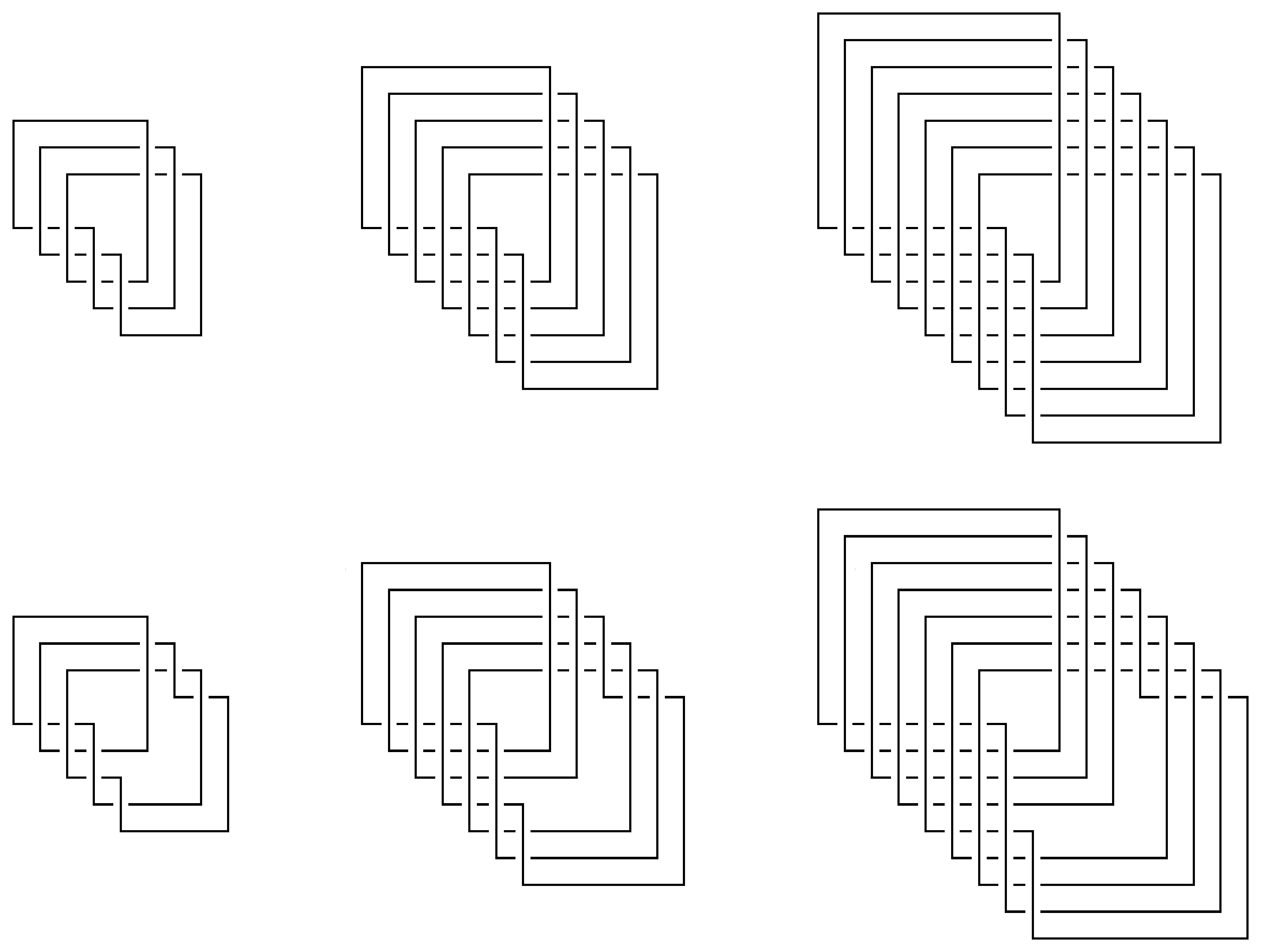}
\caption{Grid diagrams of  $T_{3,5}$, $T_{5,7}$, and $T_{7,9}$}
\label{fig:6grids}
\end{figure}

In the first row of Figure~\ref{fig:6grids}, we have minimal grid diagrams of $T_{3,5}$, $T_{5,7}$, and $T_{7,9}$.  The top-right parts of them are a positive half-twist of 3, 5, and 7 strings, respectively.  Reading counterclockwise from the bottom of the half-twists, we see that they are closures of the braids $\Delta^2\tau^2$ with $r=3,5$, and $7$, respectively.
In the second row,  we have grid diagrams with one more vertical and horizontal sticks which are  the closures of $\Delta^2\tau(\sigma_2)(\sigma_2)$,
$\Delta^2\tau(\sigma_3\sigma_4)(\sigma_4\sigma_3)$, and $\Delta^2\tau(\sigma_4\sigma_5\sigma_6)(\sigma_6\sigma_5\sigma_4)$. By Lemma~\ref{lem:conjugate}, they are the same torus knots as in the first row.

To enhance visual understanding,  we introduce some symbols to simplify grid diagrams as shown in Figure~\ref{fig:symbols}. They are used in Figure~\ref{fig:conjugate} which visualizes Lemma~\ref{lem:conjugate}.

\begin{figure}[h]
\includegraphics[width=\textwidth]{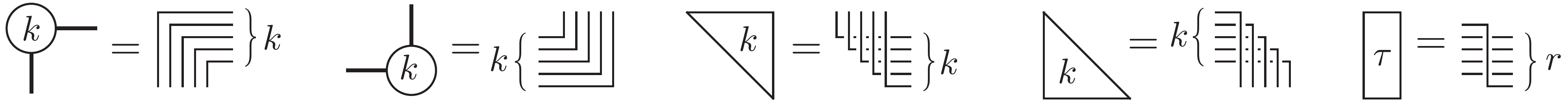}
\caption{symbolic parts of grid diagram}\label{fig:symbols}
\end{figure}

\begin{figure}[h]
\includegraphics[width=0.7\textwidth]{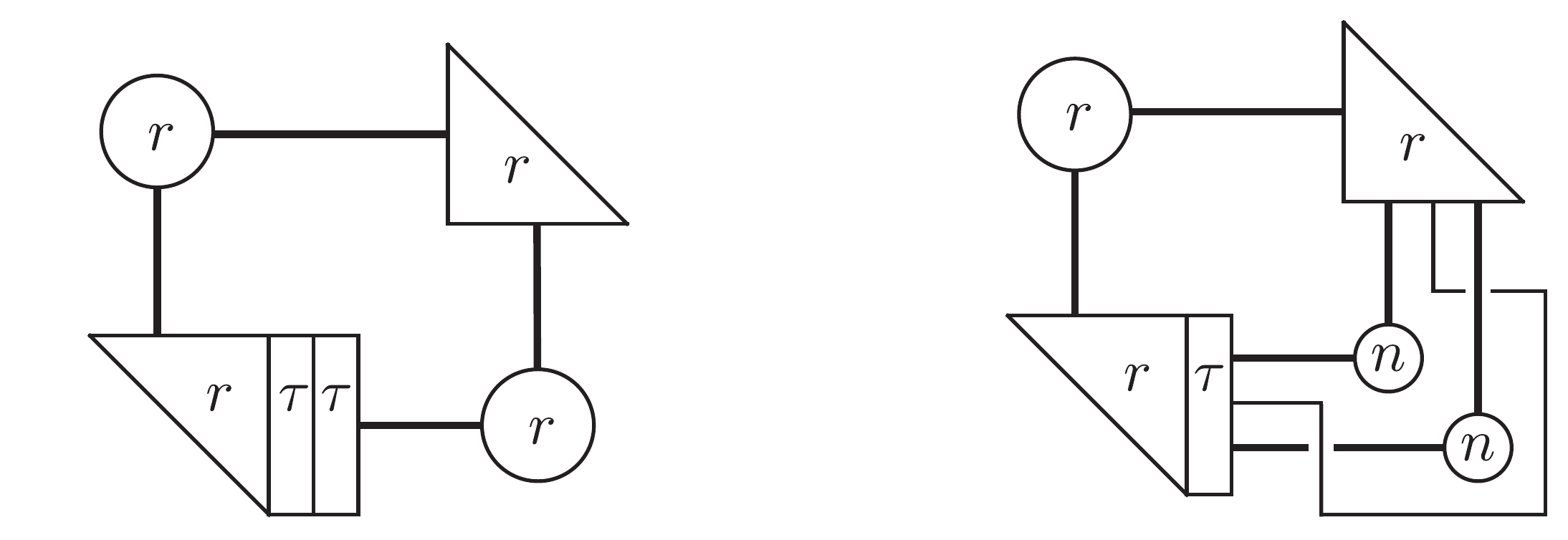}
\caption{Closures of 
$\Delta^2\tau^2$ and 
$\Delta^2\tau(\sigma_{n+1}\sigma_{n+2}\cdots\sigma_{2n})(\sigma_{2n}\sigma_{2n-1}\cdots\sigma_{n+1})$}\label{fig:conjugate}
\end{figure}

In Figure~\ref{fig:grid-moves}(a)\footnote{Same as the right part of Figure~\ref{fig:conjugate}}, we focus on the string $s$ which connect the center of $\tau$ on the right and the center of the bottom of the top-right $\Delta$.  The string $s$ is in six sticks,  horizontal and vertical. The vertical stick $s_1$ of $s$ on the right of $\tau$ is in $(r+2)$nd  position from the left and the horizontal stick $s_2$ of $s$ below $\Delta$ is in $(r+1)$st position from the top.
The string $s$ can be isotoped so that $s_1$ is moved to the right at $(n+2)$nd position from the right and then $s_2$ is moved down at $(n+3)$rd position from the bottom.  The result is shown in Figure~\ref{fig:grid-moves}(b). 
During the isotopy,  the horizontal and vertical positions of the other sticks are  adjusted accordingly.  The left half of the Figure~\ref{fig:grid-moves}(b) shows local details of that of Figure~\ref{fig:grid-moves}(a).
We move the rightmost vertical stick of $s$ in Figure~\ref{fig:grid-moves}(b) through the back of the diagram to the $(n+2)$nd position from the left. 
Then we move the bottom stick of $s$ to the top through the front of the diagram. These changes are shown  in  (c) and (d).

\begin{figure}[h]
\includegraphics[width=\textwidth]{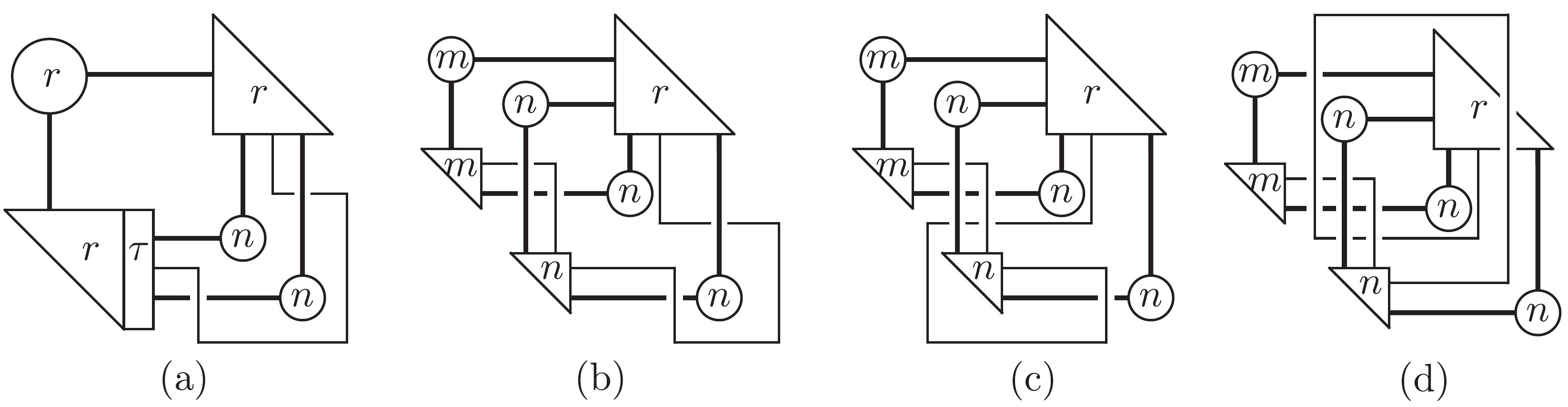}
\caption{Moving the string $s$. ($m=n+1$)}\label{fig:grid-moves}
\end{figure}

We claim that the diagram (d) is a side-view projection of a petal embedding of a $T_{r,r+2}$.  It has the following properties,  assuming that the vertical sticks are one unit apart from the neighboring ones.
\begin{enumerate}
\item There is only one vertical stick having adjacent horizontal sticks on either side. We call it the inflection stick,  denoted by $I$.  The two horizontal sticks are of length $r+1$.
\item Every vertical stick except $I$ has one adjacent horizontal stick  with length $r+1$ and the other with length $r+2$.
\end{enumerate}
These are the properties of a grid diagram associated with the  arc presentation obtained from a petal projection~\cite{Adams2015_2}.
In the diagram (d),  $I$ is the single vertical stick on top of the triangle marked with $n$. The horizontal sticks above the top of $I$ are of length $r+2$ and all others are of length $r+1$.

Now we are ready to read a petal permutation from (d).  We orient the diagram so that the top horizontal stick is oriented to the left. It is the same as the counterclockwise orientation around the obvious center.  We assign  levels $1,2,\ldots,2r+3$ to the horizontal sticks from top to bottom.  From the top horizontal stick,  read the levels of horizontal sticks as one goes along the knot.
$$\begin{aligned}
T_{r,r+2}=T_{2n+1,2n+3}&: ([1, 3n+4],[n+2,3n+3],[n+1,3n+2],\ldots,[2,2n+3],\\
&\quad [4n+5,2n+2],[4n+4,2n+1],\ldots,[3n+6,n+3],3n+5)\footnotemark
\end{aligned}$$
\footnotetext{Square brackets are inserted for the convenience of the readers.}
In the cases $n=1,2,3$,  one may read the permutations from  Figure~\ref{fig:petal_T}.

$$\begin{aligned}
T_{3,5}&: (1, 7,3,6,2,5,9,4,8)\\
T_{5,7}&: (1, 10,4,9,3,8,2,7,13,6,12,5,11)\\
T_{7,9}&: (1, 13,5,12,4,11,3,10,2,9,17,8,16,7,15,6,14)
\end{aligned}$$

\begin{figure}[h]
\includegraphics[width=0.7\textwidth]{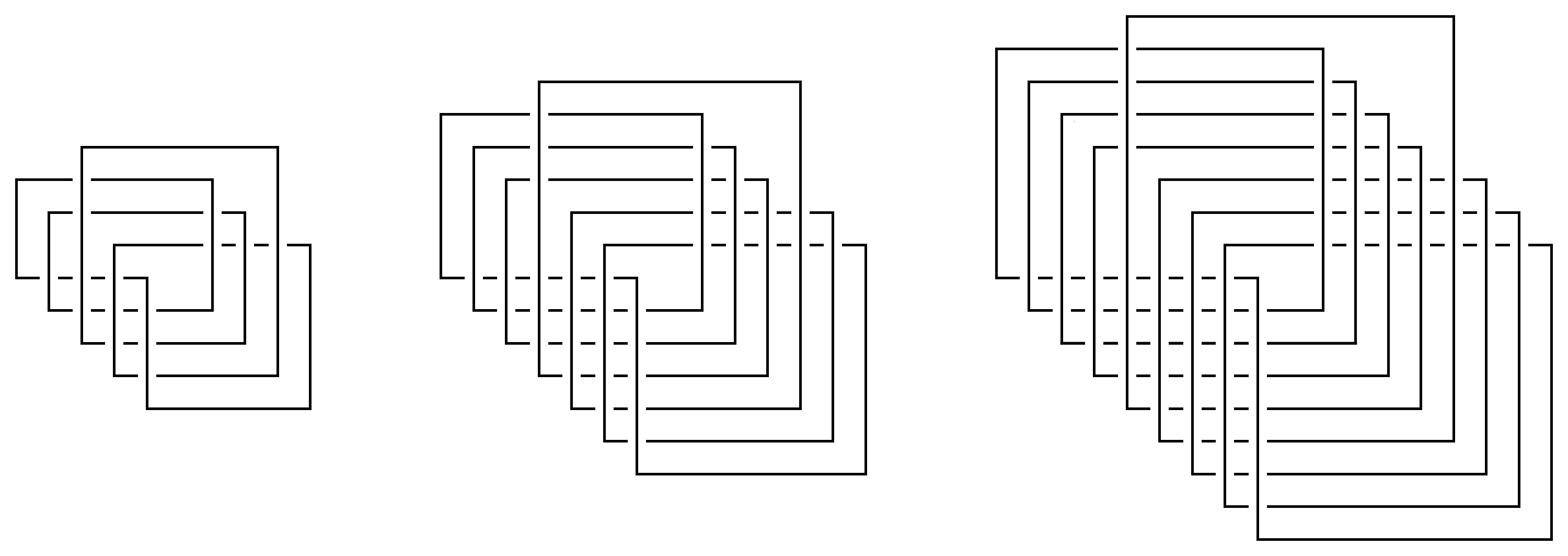}
\caption{Grid diagrams of $T_{3,5}, T_{5,7},T_{7,9}$ from which petal permutations can be read. }\label{fig:petal_T}
\end{figure}

So far we showed that $p(T_{r,r+2})\le2r+3$, completing the proof of Theorem~\ref{thm:main}.

\section*{Acknowledgements}
The first author was supported by the National Research Foundation of Korea(NRF) grant funded by the Korea government (MSIT) (2019R1A2C1005506) and the Dongguk University Research Fund.

\end{document}